\documentclass[letterpaper, 10 pt, conference]{ieeeconf}  % Comment this line out if you need a4paper

\IEEEoverridecommandlockouts                              % This command is only needed if 
\usepackage{todonotes}
\usepackage{latexsym,amssymb,amsmath, amsbsy,amsopn, amstext}
\usepackage{graphicx,epsfig,tikz,pgf,hyperref,cancel,color,float,ifpdf}
\usepackage{extarrows,mathrsfs}
\usepackage{amsmath,  amssymb,stfloats}
\usepackage{tikz,subcaption}
\usepackage[]{caption}
\usetikzlibrary{automata, positioning, arrows}
\usepackage[ruled,vlined,linesnumbered]{algorithm2e}
\usepackage{bm}
\usepackage{etoolbox}
\makeatletter
\patchcmd{\@makecaption}
  {\scshape}
  {}
  {}
  {}
\makeatother
\usepackage{makecell}

\makeatletter
\patchcmd{\@algocf@start}% <cmd>
  {-1.5em}% <search>
  {0pt}% <replace>
  {}{}% <success><failure>
\makeatother

\graphicspath{{./figures/}}

\newtheorem{definition}{\bfseries Definition}%[section]
\newtheorem{proposition}{\bfseries Proposition}%[section]
\newtheorem{example}{\bfseries Example}%[section]
\newtheorem{theorem}{\bfseries Theorem}
\newtheorem{lemma}{\bfseries Lemma}%[section]

\newtheorem{remark}{\bfseries Remark}

\newtheorem{problem}{\bfseries Problem}

\def\x{\bm{x}}

\def\u{\bm{u}}

\def\f{\bm{f}}

\newcommand{\mm}[1]{\mathcal{#1}}
\newcommand{\mt}[1]{\boldsymbol{#1}}
\newcommand{\bxi}{\bm{\xi}}
\newcommand{\z}{\bm{z}}
\newcommand{\cc}{\mathbf{c}}
\newcommand{\G}{\mathbf{G}}
\newcommand{\A}{\mathbf{A}}
\newcommand{\bb}{\mathbf{b}}

\newcommand{\Zc}{\mathcal{Z}}

\newcommand{\mat}[1]{\begin{bmatrix} #1 \end{bmatrix}}
\newcommand{\hz}[1]{\langle \G^c_{#1},\allowbreak \G^b_{#1},\allowbreak \cc_{#1},\allowbreak \A^c_{#1},\allowbreak \A^b_{#1},\allowbreak \bb_{#1} \rangle}

\allowdisplaybreaks[4]

\usepackage{xcolor}
\newif\ifdraft

\drafttrue

\title{\LARGE \bf
Backward Reachability Analysis of Neural Feedback Systems Using Hybrid Zonotopes}

\author{Yuhao Zhang, Hang Zhang and Xiangru Xu
\thanks{Y. Zhang, H. Zhang and X. Xu are with the Department of Mechanical Engineering, University of Wisconsin-Madison, Madison, WI, USA. Email:         {\tt\small \{yuhao.zhang2,hang.zhang,xiangru.xu\}@wisc.edu}.}%
}

\begin{document}
\maketitle
\begin{abstract}
The proliferation of neural networks in safety-critical applications necessitates the development of effective methods to ensure their safety. This letter presents a novel approach for computing the exact backward reachable sets of neural feedback systems based on hybrid zonotopes. It is shown that the input-output relationship imposed by a ReLU-activated neural network can be exactly described by a hybrid zonotope-represented graph set. Based on that, the one-step exact backward reachable set of a neural feedback system is computed as a hybrid zonotope in the closed form. 
%To the best of our knowledge, this is the first result that can compute exact backward reachable sets of a dynamical system that consists of a linear model and a ReLU-activated neural network controller.
%Although recent works have utilized hybrid zonotopes for the forward reachability analysis of neural feedback systems, the exponential growth of the set complexity prevents its usage on deep neural networks. In this work, algorithms with linear set complexity growth are proposed to compute the graph set of ReLU-activated neural networks and the backward reachable sets of neural feedback systems represented by hybrid zonotopes. 
In addition, a necessary and sufficient condition is formulated as a mixed-integer linear program to certify whether the trajectories of a neural feedback system can avoid unsafe regions in finite time. Numerical examples are provided to demonstrate the efficiency of the proposed approach.
\end{abstract}

\section{Introduction}\label{sec:intro}
Neural Networks (NNs) have become increasingly prevalent in autonomous systems.  %due to their effectiveness in achieving high performance. 
However, it has been shown that NNs are highly sensitive to even small perturbations in the input space, despite performing well in nominal scenarios \cite{yuan2019adversarial,goodfellow2014explaining}. %\cite{yuan2019adversarial,goodfellow2014explaining,kurakin2018adversarial}. 
%Given the potential safety risks associated with using NNs in safety-critical systems such as robotics \cite{bekey2012neural} and self-driving cars \cite{rao2018deep}, there is a pressing need for safety guarantees. 
Given the potential safety risks associated with using NNs in safety-critical systems such as robotics \cite{bekey2012neural} and self-driving cars \cite{rao2018deep}, there is a pressing need for developing efficient tools to provide safety guarantees for control systems with NN components.
%However, providing such guarantees poses significant computational challenges due to the high dimensionality and nonlinearities of NNs. Therefore, developing efficient tools that can ensure the safety of such dynamical systems with NN components is essential.

Reachability analysis of neural feedback systems, which are systems with NN controllers in the feedback loop, has been investigated in recent works  \cite{katz2017reluplex,dutta2019reachability, huang2019reachnn, everett2021reachability, tran2020nnv, Fazlyab2022Safety, zhang22constrainedzonotope, rober2022backward, rober2022hybrid, vincent2021reachable}. The majority of these results focus on forward reachability, which estimate the set of states that can be reached from an initial set \cite{katz2017reluplex,dutta2019reachability, huang2019reachnn, everett2021reachability, tran2020nnv, Fazlyab2022Safety, zhang22constrainedzonotope}. Other works such as \cite{rober2022backward, rober2022hybrid, vincent2021reachable} consider the backward reachability problem by computing a set of states, known as the Backward Reachable Set (BRS), from which 
%a proper control strategy can guide 
the system's trajectories can reach a specified \emph{target set} within a finite time. 
%Backward reachability analysis is a process that involves identifying a set of states, known as the backward reachable set (BRS), from which a proper control strategy can guide the system's trajectories towards a specified target set within a finite time. 
%It can be used to verify the safety property of systems and can be useful for identifying critical corner cases for closed-loop systems with complex controllers \cite{chou2018using, yang2021synthesis}.
% By analyzing the BRS, it is possible to determine the system's behavior under different conditions and identify potential safety hazards or performance issues. 
When the target set is the set of unsafe states, backward reachability analysis can identify the states that lead to safety violations. 
Compared to forward reachability analysis, backward reachability analysis can be more efficient in finding safety violations for a system, especially in cases where the unsafe states are rare or hard to reach from many initial states. 
% In addition, only backward reachability can lead to correct results in some cases \cite{mitchell2007comparing}.
%Backward reachability analysis can focus on the specific unsafe states that have been identified and works backward to identify the BRS that could lead to safety violations.
% This can result in a smaller set of initial states to analyze and thus reduce the computational overhead of the analysis.
Although various techniques have been developed for backward reachability analysis on systems without NNs \cite{mitchell2005time, althoff2021set, yang2022efficient}, they are not directly applicable to neural feedback systems due to the highly nonlinear and nonconvex nature of NNs. 

%Despite the early contributions above, the computation of exact BRSs of neural feedback systems has not been explored yet, to the best of our knowledge, and merits further investigation. 
%, backward reachability analysis for systems with NNs encounters notorious difficulty.
%To overcome these challenges, \cite{vincent2021reachable} represents NNs as piecewise linear functions and computes the corresponding linear regions by using the piecewise linear property of the ReLU activation function. But this method only works for the NNs in isolation. A method based on Linear Programming (LP) relaxation was proposed in \cite{rober2022backward} to compute the over-approximated BRS of \emph{neural feedback systems} (i.e., feedback systems with NN controllers). To reduce the conservatism, \cite{rober2022hybrid} extends the LP-based method with a hybrid partition strategy, which could lead to a larger computational burden for high-dimensional systems. 

This letter aims to compute the \emph{exact} BRS of a neural feedback system where the controller is a Feedforward Neural Network (FNN) with Rectified Linear Unit (ReLU) activation functions. The main mathematical tool used is \emph{Hybrid Zonotope} (HZ), which can compactly represent a finite union of polytopic sets \cite{bird2021hybrid,siefert2022robust,bird2022set,zhang2022reachability}. This work builds on our previous work \cite{zhang2022reachability}, which shows that an FNN with ReLU activation functions can be exactly represented by an HZ and provides algorithms to compute the exact and approximated forward reachable sets of neural feedback systems. The contributions of this work are at least threefold: (i) An algorithm with a linear set complexity growth rate is provided to represent the exact input-output relationship of a ReLU-activated FNN as an HZ, which is an improvement on the exponential set complexity growth rate given in \cite{zhang2022reachability}; (ii) Based on the reachability analysis of FNNs in isolation, an algorithm is proposed to compute the exact BRS of neural feedback systems represented by HZs; (iii) A necessary and sufficient condition  formulated as a Mixed-Integer Linear Program (MILP) is provided to certify the safety properties of neural feedback systems. The performance of the proposed method is demonstrated through two numerical examples.

\section{Preliminaries \& Problem Statement}\label{sec:pre}

\textbf{Notation:} 
%In this letter, vectors and matrices are denoted as bold letters (e.g., $\x \in \mathbb{R}^n$ and $\A \in \mathbb{R}^{n\times n}$). 
The $i$-th component of a vector $\x\in \mathbb{R}^n$ is denoted by $x_i$ with $i\in\{1,\dots,n\}$. For a matrix $\A\in \mathbb{R}^{n\times m}$, $\A[i:j,:]$ denotes the matrix constructed by the $i$-th to $j$-th rows of $\A$. The identity matrix is denoted as $\bm{I}$ and $\bm{e}_i$ is the $i$-th column of $\bm{I}$. The vectors and matrices whose entries are all 0 (resp. 1) are denoted as $\bm{0}$ (resp. $\bm{1}$). Given sets $\mathcal{X}\subset \mathbb{R}^n$, $\mathcal{Z}\subset \mathbb{R}^m$ and a matrix $\bm{R}\in\mathbb{R}^{m\times n}$, the Cartesian product of $\mathcal{X}$ and $\mathcal{Z}$ is $\mathcal{X}\times \mathcal{Z} = \{(\x,\z)\;|\;\x\in\mathcal{X},\z\in\mathcal{Z}\}$, the generalized intersection of $\mathcal{X}$ and $\mathcal{Z}$ under $\bm R$ is $\mathcal{X} \cap_{\bm R}\mathcal{Z} = \{\x\in\mathcal{X}\;|\;\bm R \x\in\mathcal{Z}\}$, and the $k$-ary Cartesian power of $\mathcal{X}$ is $\mathcal{X}^k = {\mathcal{X}\times\cdots\times\mathcal{X}}$.
%In this letter, scalars are denoted by lowercase letters (e.g., $\alpha \in \mathbb{R}$), vectors are denoted by bold lowercase letters (e.g., $\x \in \mathbb{R}^n$), matrices are denoted by bold capital letters (e.g., $\A \in \mathbb{R}^{n\times n}$) and sets are denoted by calligraphic letters (e.g., $\X \subset \mathbb{R}^{n\times n}$).

\subsection{Hybrid Zonotopes}\label{sec:zono}

\begin{definition}\cite[Definition 3]{bird2021hybrid}\label{def:sets}
%Let $\Zc,\Zc_c,\Zc_h \subset\mathbb{R}^n$. $\Zc$ is a \emph{zonotope} if \eqref{equ:zono} holds \cite{mcmullen1971zonotopes}, $\Zc_c$ is a \emph{constrained zonotope} if \eqref{equ:czono} holds \cite{scott2016constrained}, and $\Zc_h$ is a \emph{hybrid zonotope} if \eqref{equ:hzono} holds \cite{bird2021hybrid}:
The set $\Zc \subset\mathbb{R}^n$ is a \emph{hybrid zonotope} if there exist $\mathbf{c} \in \mathbb{R}^{n}$, $\mathbf{G}^c \in \mathbb{R}^{n \times n_{g}}$, $\G^b\in \mathbb{R}^{n \times n_{b}}$, $\A^c \in \mathbb{R}^{n_{c}\times{n_g}}$, $\A^b \in\mathbb{R}^{n_{c}\times{n_b}}$, $\bb \in \mathbb{R}^{n_{c}}$ such that
\begin{align*}
%    &\exists (\mathbf{c}, \mathbf{G}) \in  \mathbb{R}^{n}\times \mathbb{R}^{n \times n_{g}}: \!\Zc =\left\{\mathbf{G} \bm{\xi}+\mathbf{c} \;|\; \|\bm{\xi}\|_{\infty} \leq 1\right\}, \label{equ:zono}\\
%    &\exists (\mathbf{c}, \mathbf{G},\A,\bb) \in  \mathbb{R}^{n}\times \mathbb{R}^{n \times n_{g}}\times \mathbb{R}^{n_{c} \times n_{g}} \times \mathbb{R}^{n_{c}}: \notag\\
%    &  \; \quad\quad\quad\quad\quad\quad\Zc_c=\left\{\mathbf{G} \bm{\xi}+\mathbf{c} \;|\; \|\bm{\xi}\|_{\infty} \leq 1, \mathbf{A} \bm{\xi}=\mathbf{b}\right\},\label{equ:czono}\\
%    &\exists (\mathbf{c}, \mathbf{G}^c,\G^b,\A^c,\A^b,\bb) \in  \mathbb{R}^{n}\!\times\! \mathbb{R}^{n \times n_{g}}\!\times\! \mathbb{R}^{n \times n_{b}} \!\times\! \mathbb{R}^{n_{c}\times{n_g}}\notag \\
%    & \!\quad\quad\quad\quad\quad\quad\quad\quad\quad \quad\quad\times \mathbb{R}^{n_{c}\times{n_b}}\times \mathbb{R}^{n_{c}}: \label{equ:hzono}\\ \notag
    & \mathcal{Z}\!=\!\left\{\mat{\G^c \!\! & \!\!\G^b}\mat{\bm{\xi}^c\\ \bm{\xi}^b}+\cc \left |\!\! \begin{array}{c}
{\mat{\bm{\xi}^c\\ \bm{\xi}^b} \in \mathcal{B}_{\infty}^{n_{g}} \times\{-1,1\}^{n_{b}}}, \\
{\mat{\A^c \!\!& \!\!\A^b}\mat{\bm{\xi}^c\\ \bm{\xi}^b}=\bb}
\end{array}\right.\!\!\!\right\},
\end{align*}
where $\mathcal{B}_{\infty}^{n_g}=\left\{\bm{x} \in \mathbb{R}^{n_g} \;|\;\|\bm x\|_{\infty} \leq 1\right\}$ is the unit hypercube in $\mathbb{R}^{n_{g}}$. The shorthand notation of the hybrid zonotope 
%in Hybrid Constrained Generator-representation (HCG-rep) 
is given by $\mathcal{Z}= \hz{}$.
%The shorthand notations of the zonotope, constrained zonotope and hybrid zonotope  are given by $\Zc = Z\langle \cc, \G\rangle$, $\mathcal{Z}_c= CZ\langle \mathbf{c}, \mathbf{G}, \mathbf{A}, \mathbf{b}\rangle$, and $\mathcal{Z}_{h}= HZ\langle \mathbf{c}, \mathbf{G}^c, \mathbf{G}^b, \mathbf{A}^c, \mathbf{A}^b, \mathbf{b}\rangle$, respectively. 
\end{definition}

%Note that a hybrid zonotope degenerates into a constrained zonotope when $n_b=0$, and a constrained zonotope degenerates into a zonotope when $n_c=0$. 
Given an HZ $\mathcal{Z} = \hz{}$, the vector $\cc$ is called the \emph{center}, the columns of $\G^b$ are called the \emph{binary generators}, and the columns of $\G^c$ are called the \emph{continuous generators}. 
For simplicity, we define the set $\mathcal{B}(\A^c,\A^b,\bb) = \{(\bm{\xi}^c,\bm{\xi}^b) \in \mathcal{B}^{n_g}_\infty \times \{-1,1\}^{n_b} \;|\; \A^c\bm{\xi}^c + \A^b\bm{\xi}^b = \bb \}$. 

An HZ with $n_b$ binary generators is equivalent to the union of $2^{n_b}$ constrained zonotopes \cite[Theorem 5]{bird2021hybrid}. 
Identities are provided to compute the linear map and generalized intersection \cite[Proposition 7]{bird2021hybrid}, union operation \cite[Proposition 1]{bird2021unions}, and Cartesian product of HZs \cite[Proposition 3.2.5]{bird2022hybrid}.
The emptiness of an HZ can be verified by solving an MILP \cite{bird2021hybrid}.
\begin{lemma}\label{lemma:empty}
Given $\mathcal{Z} = \hz{} \subset \mathbb{R}^n$, $\mathcal{Z} \not = \emptyset$ if and only if  $\min\{\|\bm{\xi}^c\|_{\infty} \;|\; \A^c\bm{\xi}^c + \A^b\bm{\xi}^b  = \bb, \bm{\xi}^c\in \mathbb{R}^{n_g}, \bm{\xi}^b\in\{-1,1\}^{n_b}\} \leq 1.
$
\end{lemma}

\subsection{Problem Statement}\label{sec:prob}

Consider the following discrete-time linear system:
\begin{equation}\label{dt-sys}
    \x{(t+1)} = \bm A_d \x(t) + \bm B_d \u(t)% + w(t)
\end{equation}
where $\x(t)\in  \mathbb{R}^n,\; \u(t)\in \mathbb{R}^m$ are the state and the control input, respectively. %$\bm A_d \in \mathbb{R}^{n \times n}$, $\bm B_d \in \mathbb{R}^{n \times m}$ are the state matrix and the input matrix, respectively. 
%With an initial state $\x(0)$ and a control sequence $\u = \u(0), \u(1), \dots$, the state trajectory of system \eqref{dt-sys} is denoted as $\x = \x(0),\x(1),\dots$.
 We assume $\x\in\mathcal{X}$ where $\mathcal{X}\subset \mathbb{R}^n$ is called the state set and the controller is given as
$\u(t) = \pi(\x(t))$ 
%\begin{equation}
%\u(t) = \pi(\x(t)),\label{dt-input}    
%\end{equation}
where $\pi$ is an $\ell$-layer FNN with ReLU activation functions. The neural feedback system consisting of system \eqref{dt-sys} and controller $\pi$ is a closed-loop system denoted as:
\begin{equation}\label{close-sys}
    \x{(t+1)} = \f_{cl}(\x(t)) \triangleq \bm A_d\x(t) + \bm B_d \pi(\x(t)).
\end{equation}

%\begin{definition} ($t$-step backward reachable set)
Given a target set $\mathcal{T} \subset \mathcal{X}$ for the closed-loop system \eqref{close-sys}, the set of states that can be mapped into the target set $\mathcal{T}$ by \eqref{close-sys} in exactly $t$ steps is defined as the $t$-step BRS and denoted as $\mathcal{P}_t(\mathcal{T})\triangleq\{\x(0)\in \mathcal{X} | x(k)\in \mathcal{T}, x(k) = \f_{cl}(\x(k-1)),k=1,2,\dots,t\}$.
%\end{definition}
%we denote $\mathcal{P}_t(\mathcal{T})\triangleq\{\x(0)\in \mathcal{X} | x(k)\in \mathcal{T}, x(k) = \f_{cl}(\x(k-1)) ,k=1,2,\dots,t\}$ the , i.e., . 
For simplicity, the one-step BRS is also denoted as $\mathcal{P}(\mathcal{T})$, i.e., $\mathcal{P}(\mathcal{T}) = \mathcal{P}_1(\mathcal{T})$. We assume both the state set $\mathcal{X}$ and the target set $\mathcal{T}$ are represented by HZs in this work.
%for the closed-loop system \eqref{close-sys} are represented by HZs. 

For the $\ell$-layer FNN controller $\pi$, the $k$-th layer weight matrix and bias vector are denoted as $\bm W^{(k-1)}$ and $\bm v^{(k-1)}$, respectively, where $k=1,\dots,\ell$. Denote $\x^{(k)}$ as the neurons of the $k$-th layer and $n_{k}$ as the dimension of $\x^{(k)}$. Then, for $k=1,\dots,\ell-1$, we have
%The $k$-th ($k=1,\dots,\ell-1$) layer of the neuron of the FNN is given by
%for each layer $k=1,\dots,\ell-1$, the neuron of the FNN is given by
\begin{align}\label{equ:NN}
\x^{(k)}=\phi(\bm W^{(k-1)}\x^{(k-1)}+ \bm v^{(k-1)})    
\end{align}
%where $\x^{(0)} = \x(t)$ and $\phi: \mathbb{R}^{n_{k}}\rightarrow \mathbb{R}^{n_{k}}$ is the vector-valued activation function constructed by component-wise repetition of ReLU function, i.e., 
where $\x^{(0)} = \x(t)$ and $\phi$ is the \emph{vector-valued} activation function constructed by component-wise repetition of ReLU function, i.e., 
$
%\phi(\x) \triangleq [ReLU(x_1) \;\cdots\; ReLU(x_{n_{k}})]^T,\;\x\in\mathbb{R}^{n_{k}}.
\phi(\x) \triangleq [ReLU(x_1) \;\cdots\; ReLU(x_{n})]^T.
$ 
Only the linear map is applied in the last layer, i.e., $$\pi(\x(t)) = \x^{(\ell)} = \bm W^{(\ell-1)}\x^{(\ell-1)}+\bm v^{(\ell-1)}.$$ The total number of hidden neurons is denoted as $N_\pi = n_{1}+\cdots+n_{\ell-1}$.

The following problem will be investigated in this work.

\begin{problem}\label{prob1}
    Given a target set $\mathcal{T}\subset \mathcal{X}$ represented as an HZ and a time horizon $T\in \mathbb{Z}_{>0}$, compute the exact BRS $\mathcal{P}_t(\mathcal{T})$ of the neural feedback system \eqref{close-sys}, for $t=1,2,\dots,T$.
\end{problem}

\section{Exact Backward Reachability Analysis}\label{sec:main}

In this section, we first present a technique that can represent the exact input-output relationship of a ReLU-activated FNN as an HZ-based graph set that has a linear set complexity growth rate. 
%construct an HZ to represent the input-output relationship of a ReLU-activated FNN exactly. 
Then, based on that, we show if the target set is given as an HZ, the exact BRS of the system \eqref{close-sys} can  be also represented as HZs in closed form.% through simple identities.

%In order to solve Problem \ref{prob1}, we first consider the problem of computing a one-step BRS of the closed-loop system \eqref{close-sys}. From the definition in Section \ref{sec:prob}, we know 
%\begin{align*}
 %   Pre(\mathcal{T}) = & \{\x\in\mathcal{X}\;|\; A_d\x+B_d\u \in \mathcal{T}, \u = \pi(\x)\}.\\
%\end{align*}
%Inspired by the invariant set computation methods proposed in \cite{keerthi1987computation,anevlavis2019computing} that involve lifting
%a polytope to a higher-dimensional space followed by projection, we will first compute the one-step BRS in a lifted space and then project it back onto the state space.

\subsection{Representation of the Graph of FNNs via HZs} \label{sec:fnn}

 The problem of computing the BRS and invariant set of controlled dynamical systems has been studied in many works, such as \cite{keerthi1987computation,blanchini2008set,anevlavis2019computing}. A commonly-used technique is to abstract the constraints imposed by the dynamic system in the input-output space. For neural feedback systems, the imposed constraints can be identified by finding a proper representation of the input-output relationship of the NN controllers.

%It has been shown in \cite{zhang2022reachability} that when the input set to an FNN is a hybrid zonotope, the output set can be exactly represented by a single hybrid zonotope. However, in the worst case, the complexity of the computed hybrid zonotope in \cite{zhang2022reachability} will grow exponentially with the number of neurons in the FNN, which makes it unsuitable for the analysis of deep neural networks. In this subsection, we present a novel approach that constructs the exact hybrid zonotope representation of FNN with linear complexity. 

One of the major difficulties in analyzing NNs is the composition of nonlinear activation functions \cite{Fazlyab2022Safety}. To simplify the analysis of NNs, quadratic constraints have been utilized to abstract the constraints imposed by the NNs on the pre- and post-activation signals \cite{Fazlyab2022Safety,yin2021stability}. For ReLU-activated FNNs, different types of methods are also proposed to abstract the nonlinear functions with linear constraints  \cite{rossig2021advances}. Building upon these methodologies, our approach employs an HZ to capture the constraints imposed by NNs in an exact manner. Specifically, we denote
$$
\mathcal{G}(\pi,\mathcal{X}) = \{(\x,\u)\;|\; \u = \pi(\x),\x\in\mathcal{X}\} \subset \mathbb{R}^{n+m}
$$ 
as the \emph{graph} of the ReLU-activated FNN $\pi$ over the state space domain $\mathcal{X}$, and we will show that there exists an HZ $\mathcal{H}_\pi = \hz{\pi}$ such that $\mathcal{G}(\pi,\mathcal{X}) = \mathcal{H}_\pi$. 
%describe the that contains all the possible input-output pairs of a ReLU-activated FNN over a region of interest. 
%This approach is inspired by the usage of state-update sets represented as hybrid zonotopes in \cite{siefert2022robust} to represent all possible state transitions of a dynamic system.

%\begin{definition}
%    Let $\bm{\varphi}: \mathbb{R}^{n_i} \rightarrow \mathbb{R}^{n_o}$ be an arbitrary function and $\mathcal{S}\subseteq \mathbb{R}^{n_i}$ be a subset of its domain. The \emph{input-output set} of the function $\bm{\varphi}$ over the input set $\mathcal{S}$ is defined as 
%    \begin{align}\label{equ:io-set}
%        \Phi(\bm{\varphi},\mathcal{S}) = \{(\x,\mathbf{y}) \in \mathbb{R}^{n_i} \times \mathbb{R}^{n_o}\;|\; \mathbf{y} = \bm{\varphi(\x)}, \x\in \mathcal{S}\}.
%    \end{align}
%\end{definition}

To that end, we first consider the representation of a scalar-valued ReLU function $x = ReLU(z) = \max\{z,0\}$ over an interval domain $[-\alpha,\beta]$ where $\alpha,\beta\in\mathbb{R}_{>0}$. The graph of the ReLU function over the interval domain is plotted in Fig. \ref{fig:relu}.

\begin{figure}[h]
    \centering
         \includegraphics[width=0.30\textwidth]{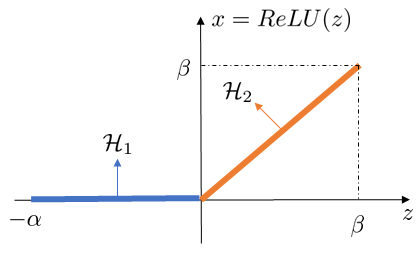}
     \caption{The graph of the ReLU function as the union of two HZs $\mathcal{H}_1$ and $\mathcal{H}_2$, over the interval domain $[-\alpha,\beta]$.}
    \label{fig:relu}
    \vspace{-0.35cm}
\end{figure}

It is obvious that the set of points satisfying the ReLU function over $[-\alpha,\beta]$ form two line segments which can be exactly represented as two HZs $\mathcal{H}_1$ and $\mathcal{H}_2$ given as follows:
\begin{align*}
    \mathcal{H}_1 \!=\! \left\langle\!  \mat{\frac{\alpha}{2}\\0},\emptyset, \mat{\frac{-\alpha}{2}\\0},\emptyset,\emptyset,\emptyset \!\right\rangle,\! \mathcal{H}_2 \!=\!  \left\langle\!  \mat{\frac{\beta}{2}\\\frac{\beta}{2}},\emptyset,\mat{\frac{\beta}{2}\\\frac{\beta}{2}},\emptyset,\emptyset,\emptyset \!\right\rangle\!.
\end{align*}

Using Proposition 1 in \cite{bird2021unions}, the union of $\mathcal{H}_1$ and $\mathcal{H}_2$ can be directly computed as
\begin{equation}\label{equ:H}
    \mathcal{H} = \mathcal{H}_1 \cup \mathcal{H}_2 = \hz{h}
\end{equation}
where
\begin{align*}
    &  \G^c_h \! =\! \mat{\frac{\alpha}{2} & \frac{\beta}{2} & 0& 0&0&0\\ 0 & \frac{\beta}{2} & 0& 0&0&0}, \G^b_h \! =\! \mat{-\frac{\alpha+\beta}{4}\\-\frac{\beta}{4}}, \cc_h\! =\! \mat{\frac{\beta-\alpha}{4} \\ \frac{\beta}{4}},\\
    & \A^c_h \! =\!\mat{1& 0 & 1 & 0 & 0 & 0\\-1& 0 & 0 & 1 & 0 & 0\\0& 1 &0 & 0 & 1 & 0\\0& -1 & 0 & 0 & 0 & 1}, \A^b_h \! =\! \mat{\frac{1}{2}\\\frac{1}{2}\\-\frac{1}{2}\\-\frac{1}{2}}, \bb_h \! =\!\mat{\frac{1}{2}\\\frac{1}{2}\\\frac{1}{2}\\\frac{1}{2}}.
\end{align*}

Using \eqref{equ:H}, we get the exact representation of the graph of the ReLU function over $[-\alpha,\beta]$ using HZ, i.e.,
$
    \mathcal{H} = \{(z,x)\in \mathbb{R}^2\;|\;x = ReLU(z), z\in [-\alpha,\beta]\}.
$
Note that the graph of a ReLU function can be also linearly approximated using intervals, symbolic intervals, and polytopes, as stated in \cite{rossig2021advances}; however, the nonlinear nature of the ReLU function makes it impossible for these convex relaxation-based set representations to exactly represent its graph.
%Note that $\mathcal{H}$ has 6 continuous generators, 1 binary generator, and 4 equality constraints. 

In the following lemma, the analysis described above on the ReLU function is extended to the vector-valued activation function $\phi$ over a domain represented as an HZ.

\begin{lemma}\label{lemma:phi}
    Given a domain represented as an HZ $\mathcal{Z}\subset \mathbb{R}^{n_{k}}$, the graph of the $k$-th layer's vector-valued activation function $\phi: \mathbb{R}^{n_{k}}\rightarrow \mathbb{R}^{n_{k}}$ over $\mathcal{Z}$ can be exactly represented by the following HZ: 
    \begin{equation}\label{equ:g_phi}
        \mathcal{G}(\phi,\mathcal{Z}) = (\bm{P}\cdot\mathcal{H}^{n_{k}})\cap_{[\bm{I}\; \bm{0}]} \mathcal{Z}
    \end{equation}
    where $\bm{P}=[\bm{e}_{2}\; \bm{e}_{4}\;\cdots\; \bm{e}_{2n_{k}} \; \bm{e}_{1}\; \bm{e}_{3}\;\cdots\; \bm{e}_{2n_{k}-1}]^T\in \mathbb{R}^{2n_{k}\times 2n_{k}}$ is a permutation matrix and $\mathcal{H}$ is given in \eqref{equ:H}.
    %Given the vector-valued ReLU activation function $\phi: \mathbb{R}^{n^{(k)}}\rightarrow \mathbb{R}^{n^{(k)}}$ and a hybrid zonotope domain $\mathcal{Z}\subset \mathbb{R}^{n^{(k)}}$, the graph $\mathcal{G}(\phi,\mathcal{Z})$ of the function $\phi$ over the domain $\mathcal{Z}$ can be exactly represented by a hybrid zonotope. 
\end{lemma}
\begin{proof}
Since the HZ $\mathcal{Z}$ is a closed set, we can always find large enough scalars $\alpha,\beta\in \mathbb{R}_{>0}$ such that the interval $\mathcal{I} = [-\alpha \bm{1},\beta \bm{1}] \subset \mathbb{R}^{n_{k}}$ is an enclosure of $\mathcal{Z}$, i.e., $\mathcal{Z}\subseteq \mathcal{I}$.

Let $\z^{(k)}$ denote the input of function $\phi$ and $\x^{(k)}$ denote the output. The graph of $\phi$ over the domain $\mathcal{I}$ is $\mathcal{G}(\phi,\mathcal{I})=\{(\z^{(k)},\x^{(k)})\;|\; \x^{(k)} = \phi(\z^{(k)}),\z^{(k)}\in\mathcal{I}\}\subset \mathbb{R}^{2n_{k}}$. As the vector-valued activation function $\phi$ is constructed by component-wise repetition of ReLU functions, i.e., $x^{(k)}_i = ReLU(z^{(k)}_i)$, we have
$
    [z^{(k)}_1\;x^{(k)}_1\;z^{(k)}_2\;x^{(k)}_2\;\cdots\;z^{(k)}_{n_{k}}\;x^{(k)}_{n_{k}}]^T \in \mathcal{H}^{n_{k}} \subset \mathbb{R}^{2n_{k}}.
$

To reassemble the pairs of input and output elements in the same order of $[{\z^{(k)}}^T,{\x^{(k)}}^T]^T$, we use the permutation matrix $\bm{P}$ and get $[{\z^{(k)}}^T,{\x^{(k)}}^T]^T = [z^{(k)}_1\cdots z^{(k)}_{n_{k}}\;x^{(k)}_1\cdots x^{(k)}_{n_{k}}]^T = \bm{P} [z^{(k)}_1\;x^{(k)}_1\;\cdots\;z^{(k)}_{n_{k}}\;x^{(k)}_{n_{k}}]^T$. 

Since HZs are closed under the linear map and generalized intersection \cite[Proposition 7]{bird2021hybrid}, the graph of $\phi$ over the interval $\mathcal{I}$ is an HZ as $\mathcal{G}(\phi,\mathcal{I}) = \bm{P}\cdot\mathcal{H}^{n_{k}}$. Then, we have 
$\mathcal{G}(\phi,\mathcal{Z}) = \{(\z^{(k)},\x^{(k)})\;|\; \x^{(k)} = \phi(\z^{(k)}),\z^{(k)}\in\mathcal{Z}\} = \mathcal{G}(\phi,\mathcal{I}) \cap_{[\bm{I}\; \bm{0}]} \mathcal{Z} = (\bm{P}\cdot\mathcal{H}^{n_{k}})\cap_{[\bm{I}\; \bm{0}]} \mathcal{Z}$
%\begin{align*}
%    \mathcal{G}(\phi,\mathcal{Z}) & = \{(\z^{(k)},\x^{(k)})\;|\; \x^{(k)} = \phi(\z^{(k)}),\z^{(k)}\in\mathcal{Z}\}\\
%    &= \mathcal{G}(\phi,\mathcal{I}) \cap_{[\bm{I}\; \bm{0}]} \mathcal{Z}\\
%    &= (\bm{P} \mathcal{H}^{n_{k}})\cap_{[\bm{I}\; \bm{0}]} \mathcal{Z}
%\end{align*}
%$\mathcal{G}(\phi,\mathcal{Z}) = \{(\z^{(k)},\x^{(k)})\;|\; \x^{(k)} = \phi(\z^{(k)}),\z^{(k)}\in\mathcal{Z}\} = \mathcal{G}(\phi,\mathcal{I}) \cap_{[\bm{I}\; \bm{0}]} \mathcal{Z} = (\bm{P} \mathcal{H}^{n^{(k)}})\cap_{[\bm{I}\; \bm{0}]} \mathcal{Z}$, 
, which is also an HZ. This completes the proof.
\end{proof}

\begin{remark}
%Different types of quadratically constrained sets were derived in \cite{Fazlyab2022Safety} to over-approximate the graph of the nonlinear activation function $\phi$. 
Lemma \ref{lemma:phi} shows that the graph set of a vector-valued ReLU activation function can be exactly represented by an HZ. Lemma 4 in \cite{Fazlyab2022Safety} abstracts the input-output relationship of the ReLU function using quadratic constraints. However, their proposed approach will only provide an over-approximation of the graph set.
%For a ReLU activation function, the HZ computed by Lemma \ref{lemma:phi} can exactly represent the graph set, while Lemma 4 in \cite{Fazlyab2022Safety} only provides an over-approximation of the graph set using quadratic constraints.
% Compared to the existing results, the constructed hybrid zonotope in Lemma \ref{lemma:phi} represents the graph of $\phi$ exactly, when ReLU is chosen as the activation function and the domain $\mathcal{Z}$ can be represented by a hybrid zonotope.
\end{remark}

From the structure of the FNN $\pi$ in \eqref{equ:NN}, it is obvious that each layer is a composition of the activation function $\phi$ and the linear map with weight matrix $\bm{W}$ and bias vector $\bm{v}$. Therefore, to construct the HZ representation $\mathcal{H}_\pi = \mathcal{G}(\pi,\mathcal{X})$ for the graph of the entire network $\pi$, we can repeat the procedures described in Lemma \ref{lemma:phi} layer-by-layer and connect the input of the $k$-th layer $\z^{(k)}$ and the output of the $(k-1)$-th layer $\x^{(k-1)}$ with the linear map $\z^{(k)} = \bm{W}^{(k-1)} \x^{(k-1)}+\bm{v}^{(k-1)}$. The details on the iterative construction of the HZ $\mathcal{H}_\pi$ are summarized in Algorithm \ref{alg:1}.

\begin{algorithm}[!ht]
\SetNoFillComment
\caption{Exact graph set computation of FNN via HZs}\label{alg:1}
\KwIn{HZ domain $\mathcal{X}$, number of layers $\ell$, weight matrices $\{\bm W^{(k-1)}\}_{k=1}^{\ell}$, bias vectors $\{\bm v^{(k-1)}\}_{k=1}^{\ell}$, large scalars $\alpha,\beta >0$}
\KwOut{exact graph set as an HZ $\mathcal{H}_\pi = \mathcal{G}(\pi,\mathcal{X})$}
%\tcc{Initialization}
$\mathcal{X}^{(0)}$ $\leftarrow$ $\mathcal{X} = \hz{x}$\\
$\mathcal{H}$ $\leftarrow$ compute the graph of ReLU using \eqref{equ:H}\\
%\tcc{Iterate over each layer}
\For{$k \in \{1,2,\dots,\ell-1$\}}{
%\tcp{Get input set via linear map}
$\mathcal{Z}^{(k-1)}\!\leftarrow\!\bm{W}^{(k-1)}\mathcal{X}^{(k-1)}\!+\!\bm{v}^{(k-1)}$\tcp*{Input set}
%\tcp{Get graph set using \eqref{equ:g_phi}}
$\mathcal{G}^{(k)}$ $\leftarrow$ $(\bm{P}\cdot\mathcal{H}^{n_{k}})\cap_{[\bm{I}\;\bm{0}]} \mathcal{Z}^{(k-1)}$\tcp*{Using \eqref{equ:g_phi}}
%\tcp{Project onto the output set}
$\mathcal{X}^{(k)}$ $\leftarrow$ $[\bm{0}\;\bm{I}] \cdot \mathcal{G}^{(k)}$ \tcp*{Output set}
}
%\tcp{Linear map of last layer}
$\mathcal{X}^{(\ell)}$ $\leftarrow$ $\bm{W}^{(\ell-1)}\mathcal{X}^{(\ell-1)}+\bm{v}^{(\ell-1)}$\tcp*{Last layer}
$\hz{}$ $\leftarrow$ $\mathcal{X}^{(\ell)}$\\
\tcp{Stack input and output}
$\mathcal{H}_\pi$ $\leftarrow$ $\langle \mat{\G^c_x \!&\! \bm{0}\\ \multicolumn{2}{c}{\G^c}},\mat{\G^b_x \!&\! \bm{0}\\ \multicolumn{2}{c}{\G^b}},\mat{\cc_x\\ \cc},\A^c,\A^b,\bb\rangle$\\
\Return{$\mathcal{H}_\pi$}
\end{algorithm}

\begin{theorem}
Given an $\ell$-layer ReLU-activated FNN $\pi:\mathbb{R}^n \rightarrow \mathbb{R}^m$ and an HZ $\mathcal{X}\subset\mathbb{R}^n$, the output of Algorithm \ref{alg:1} $\mathcal{H}_\pi$ is an HZ that can exactly represent the graph set of $\pi$ over the domain $\mathcal{X}$, i.e. $\mathcal{H}_\pi = \mathcal{G}(\pi,\mathcal{X})$.
%The input-output relationship of a ReLU-activated FNN can be exactly represented by a hybrid zonotope. 
\end{theorem}
\begin{proof}
%From Algorithm \ref{alg:1}, the output $\mathcal{H}_\pi$ is constructed by using the linear map, generalized intersection, and Cartesian production of hybrid zonotopes. Thus, we know $\mathcal{H}_\pi$ is a hybrid zonotope based on Lemma \ref{lemma:set-op}. 
For the $\ell$-layer ReLU-activated FNN $\pi$, it is easy to check that the input set $\mathcal{Z}^{(k-1)}$, graph set $\mathcal{G}^{(k)}$ and output set $\mathcal{X}^{(k)}$ of the $k$-th layer activation function $\phi$ are computed iteratively for $k=1,\dots,\ell-1$ in Line 4-6 of Algorithm \ref{alg:1}. For the last layer, only a linear map is applied and the output set of FNN $\pi$ is computed as $\mathcal{X}^{(\ell)}$ in Line 7. Note that from the construction, the equality constraints in the domain set $\mathcal{X}$ are included in $\mathcal{X}^{(\ell)}$. Therefore, in Line 8, $\mathcal{H}_\pi$ stacks the input and output of $\pi$ as $\mathcal{H}_\pi = \{(\x,\u)\;|\;\x\in\mathcal{X},\u = \pi(\x)\} = \mathcal{G}(\pi,\mathcal{X})$, which is an exact representation of the graph set of $\pi$ over $\mathcal{X}$.
\end{proof}

Denote $n_{g,x}$, $n_{b,x}$ and $n_{c,x}$ as the number of continuous generators, binary generators and equality constraints of the HZ $\mathcal{X}$, respectively. The set complexity growth of the graph set $\mathcal{H}_\pi$ is given by
$
    n_{g,\pi}  = n_{g,x} + 6 N_\pi,\; n_{b,\pi} = n_{b,x} + N_\pi,\;
    n_{c,\pi} = n_{c,x} + 5 N_\pi.
$ 
%\begin{remark} \label{rmk:complexity}
The output set $\mathcal{X}^{(\ell)}$ of the FNN $\pi$ computed in Algorithm \ref{alg:1} has the same set complexity as $\mathcal{H}_\pi$. In our previous work \cite{zhang2022reachability}, it has been shown that a ReLU-activated FNN can be exactly represented by an HZ; however, the set complexity of the computed HZ there will grow exponentially with the number of neurons in the FNN. In comparison. the HZ representation of FNN produced by Algorithm \ref{alg:1} has a linear set complexity growth rate, which makes it more applicable to deep neural networks as demonstrated in Section \ref{sec:sim}.
%\end{remark}

\subsection{Computation of Exact BRS for Neural Feedback Systems} \label{sec:back}

In this subsection, we will consider the computation of exact BRS for the neural feedback system \eqref{close-sys}. Given a target set represented by an HZ, $\mathcal{T}$, the following theorem provides the closed-form of the one-step BRS, $\mathcal{P}(\mathcal{T})$.

%From the previous subsection, we can get the hybrid zonotope:
%\begin{align*}
%    \mathcal{NN} = & \{\mat{\x \\\u} \;|\; \u = \pi(\x), \x\in \mathcal{X}\}\\
%    = & HZ\langle \cc_\pi,\G_\pi^c,\G^b_\pi,\A_\pi^c,\A_\pi^b,\bb_\pi \rangle .
%\end{align*}

\begin{theorem}\label{thm:brs}
    Given any HZ $\mathcal{X} \subset\mathbb{R}^n$, let $\mathcal{H}_\pi = \hz{\pi}$ be the computed graph set of the FNN $\pi$ over the domain $\mathcal{X}$ using Algorithm \ref{alg:1}, i.e. $\mathcal{H}_\pi = \mathcal{G}(\pi,\mathcal{X})$. Let $\bm D = \mat{\bm A_d & \bm B_d}$. Then, for any target set represented by an HZ $\mathcal{T}= \hz{\tau} \subset\mathbb{R}^n$, the one-step BRS of the neural feedback system \eqref{close-sys} is an HZ given as
    \vspace{-0.1cm}
    \begin{equation}
        \mathcal{P}(\mathcal{T}) = \hz{p}
    \vspace{-0.1cm}
    \end{equation}
    where
    \vspace{-0.1cm}
    \begin{align*}
         \G^c_p &= \mat{\G^c_\pi[1:n,:] & \bm{0}},\; \G^b_p = \mat{\G^b_\pi[1:n,:] & \bm{0}},\\
         \A^c_p &= \mat{\A^c_\pi & \bm{0}\\ \bm{0} & \A^c_\tau \\ \bm{D} \G^c_\pi & -\G^c_\tau },\; \A^b_p = \mat{\A^b_\pi & \bm{0}\\ \bm{0} & \A^b_\tau\\ \bm{D} \G^b_\pi & -\G^b_\tau},\\
         \cc_p &= \cc_\pi[1:n,:],\;\; \bb_p = \mat{\bb_\pi\\\bb_\tau\\\cc_\tau- \bm{D} \cc_\pi}.
    \end{align*}
\end{theorem}

\begin{proof}
    %The one-step BRS can be constructed as follows:
    By the definition of the one-step BRS, we have
        \begin{align*}
    \mathcal{P}(\mathcal{T}) & = \{\x \in\mathcal{X}\;|\; \bm f_{cl}(\x) \in \mathcal{T}\}\\ 
    & = \{\x \;|\; \bm A_d\x+\bm B_d\u\in \mathcal{T}, \u = \pi(\x), \x\in \mathcal{X}\}\\
    & = \{\x \;|\; \bm D [\x^T \; \u^T]^T\in \mathcal{T}, [\x^T \; \u^T]^T \in \mathcal{H}_\pi\}.
\end{align*}

Since $\mathcal{T} = \{\cc_\tau+\G^c_\tau\bxi_\tau^c+ \G^b_\tau\bxi^b_\tau\;|\; (\bxi_\tau^c,\bxi_\tau^b)\in \mathcal{B}(\A^c_\tau,\A^b_\tau,\bb_\tau)\}$ and $\mathcal{H}_\pi = \{\cc_\pi+\G^c_\pi\bxi_\pi^c+ \G^b_\pi\bxi^b_\pi\;|\; (\bxi_\pi^c,\bxi_\pi^b)\in \mathcal{B}(\A^c_\pi,\A^b_\pi,\bb_\pi)\}$, we get
$\bm{D}[\x^T \; \u^T]^T = \cc_\tau+\G^c_\tau\bxi_\tau^c+ \G^b_\tau\bxi^b_\tau$ and $[\x^T \; \u^T]^T = \cc_\pi+\G^c_\pi\bxi_\pi^c+ \G^b_\pi\bxi^b_\pi$. 

Using $\x = [\bm{I}_n \; \bm{0}]\cdot [\x^T \; \u^T]^T$, we have 
$
\mathcal{P}(\mathcal{T})= \{[\bm I_n \; \bm{0}](\cc_\pi+\G^c_\pi\bxi_\pi^c+ \G^b_\pi\bxi^b_\pi)\; |\; \bm D (\cc_\pi+\G^c_\pi\bxi_\pi^c+ \G^b_\pi\bxi^b_\pi)  = \cc_\tau+\G^c_\tau\bxi_\tau^c+ \G^b_\tau\bxi^b_\tau, (\bxi_\tau^c,\bxi_\tau^b)\in \mathcal{B}(\A^c_\tau,\A^b_\tau,\bb_\tau), (\bxi_\pi^c,\bxi_\pi^b)\in \mathcal{B}(\A^c_\pi,\A^b_\pi,\bb_\pi) \}.
$ Let $\bxi^c = [(\bxi^{c}_\pi)^T \; (\bxi^c_\tau)^T]^T$ and $\bxi^b = [(\bxi^b_\pi)^T\;(\bxi^b_\tau)^T]^T$. Then,
    \begin{align*}
        \mathcal{P}(\mathcal{T}) = & \{\mat{\G^c_\pi[1\!:\!n,:] \!&\! \bm{0}}\bxi^c+ \mat{\G^b_\pi[1\!:\!n,:] \!&\! \bm{0}}\bxi^b+\\
    &\cc_\pi[1\!:\!n,:]\; |\; (\bxi^c,\bxi^b)\in \mathcal{B}(\mat{\A^c_\pi & \bm{0}\\ \bm{0} & \A^c_\tau \\ \bm D \G^c_\pi & -\G^c_\tau } ,\\
    & \mat{\A^b_\pi & \bm{0}\\ \bm{0} & \A^b_\tau\\ \bm D \G^b_\pi & -\G^b_\tau },\mat{\bb_\pi\\\bb_\tau\\\cc_\tau- \bm D \cc_\pi}) \} \\
     =&  \hz{p}.
    \end{align*}
%    \vspace{-0.8cm}
\end{proof}

To the best of our knowledge, Theorem \ref{thm:brs} is the first result that can compute the exact BRS of a neural feedback system that consists of a linear model and an FNN controller.

Based on Theorem \ref{thm:brs}, the exact $T$-step BRS of  system \eqref{close-sys} can be computed iteratively as follows:
\vspace{-0.1cm}
\begin{equation}\label{equ:t-step}
    \begin{aligned}
    \mathcal{P}_0(\mathcal{T}) =  \mathcal{T},\quad \mathcal{P}_t(\mathcal{T}) 
    %= &\{x \;|\; A_d \x+B_d\pi(\x)\in \mathcal{P}_{t-1}(\mathcal{T}), \mat{\x\\\u} \in \mathcal{F}\}\\
    = \mathcal{P}(\mathcal{P}_{t-1}(\mathcal{T})),\; t=1,\dots,T.
\end{aligned}
%\vspace{-0.1cm}
\end{equation}

Assuming that the target set $\mathcal{T}$ has $n_{g,\tau}$ continuous generators, $n_{b,\tau}$ binary generators and $n_{c,\tau}$ equality constraints, the set complexity of the $T$-step BRS computed using \eqref{equ:t-step} and Theorem \ref{thm:brs} is given by $n_{g,p} = T\cdot(n_{g,x}+6N_\pi) + n_{g,\tau}$, $n_{b,p} = T\cdot(n_{b,x}+N_\pi) + n_{b,\tau}$, $n_{c,p} = T\cdot(n_{c,x}+5N_\pi+n) + n_{c,\tau}$, where the subscript $p$ represents the one-step BRS $\mathcal{P}(\mathcal{T})$.

\begin{remark}
    In \cite{rober2022backward}, an algorithm was proposed to over-approximate the BRS of a neural feedback system using convex relaxation of NNs. The result was generalized in \cite{rober2022hybrid} where a hybrid partition scheme was presented to reduce the conservatism induced by the relaxation. Note that the BRSs computed in these works are inexact. 
    %However, in this work, the graph sets of ReLU-activated FNNs are constructed using HZs which preserve the exact input-to-output relationship imposed by the neural network. And this preserved relationship allows us to compute the exact BRS in Theorem \ref{thm:brs}, given the desired target set. 
    In \cite{vincent2021reachable}, a method was  presented to compute the exact BRS of a ReLU-activated FNN by determining the activation pattern, but it is only applicable to NNs in isolation, not neural feedback systems.
\end{remark}

\subsection{Extension to Saturated Control Input Case }
The analysis in the preceding subsections can be readily extended to neural feedback systems with saturated control inputs, using techniques similar to \cite{everett2021reachability}.
%, our approach can be readily extended to cases when. 
Specifically, assume that  
%is based on the assumption that there don't exist any control input constraints. 
%Note that real-world actuators usually have physical limits that can only employ bounded control inputs. 
the system \eqref{dt-sys} has interval control input constraints, i.e. $\u\in\mathcal{U}=[\underline{\u},\overline{\u}]$. Then the closed-loop system \eqref{close-sys} becomes
    \begin{equation}\label{close-sys-ulimit}
        \x(t+1) = \bm A_d \x(t) + \bm B_d\;sat_{\underline{\u}}^{\overline{\u}}(\pi (\x(t)))
    \end{equation}
    where the saturation function %is applied component-wisely to the control input, i.e., 
    %\begin{equation}
     %   sat_{\underline{u}_i}^{\overline{u}_i}(u_i)\! \triangleq\! \left\{ \begin{array}{ll}
     %        \!\!\underline{u}_i&  \text{if } \underline{u}_i \leq u_i\\
     %        \!\!u_i&  \text{if } \underline{u}_i < u_i < \overline{u_i}\\
     %        \!\!\overline{u}_i & \text{if } u_i \leq \overline{u}_i
     %   \end{array}, \right.\; \forall i= 1,\dots,m.
    %\end{equation}
    %Clearly, the saturation function 
    can be equivalently described by the ReLU functions as $sat_{\underline{\u}}^{\overline{\u}}(\u) = \min\{\max\{\u,\overline{\u}\},\underline{\u}\} = ReLU(-ReLU(\overline{\u}-\u)+\underline{\u}-\overline{\u})+\underline{\u}$. Therefore, the saturated NN controller $\hat{\pi}(\x) = sat_{\underline{\u}}^{\overline{\u}}(\pi (\x))$ is an $(\ell+2)$-layer ReLU-activated FNN. Denote the $k$-th layer weight matrix and bias vector of $\hat{\pi}$ as $\hat{\bm{W}}^{(k-1)}$ and $\hat{\bm{v}}^{(k-1)}$, respectively. Then, for the last three layers we have 
    \begin{equation*}
        \begin{aligned}
        & \hat{\bm{W}}^{(\ell-1)} = -{\bm{W}}^{(\ell-1)},\; \hat{\bm{W}}^{(\ell)} = -\bm{I},\; \hat{\bm{W}}^{(\ell+1)} = \bm{I},\\
        &  \hat{\bm{v}}^{(\ell-1)} = \overline{\u} - {\bm{v}}^{(\ell-1)},\; \hat{\bm{v}}^{(\ell)} = \overline{\u} - \underline{\u},\; \hat{\bm{v}}^{(\ell+1)} = \underline{\u}.\\
        %& \hat{\bm{W}}^{(\ell-1)} = -{\bm{W}}^{(\ell-1)},\; \hat{\bm{v}}^{(\ell-1)} = \overline{\u} - {\bm{v}}^{(\ell-1)},\\
        %& \hat{\bm{W}}^{(\ell)} = -\bm{I},\; \hat{\bm{v}}^{(\ell)} = \overline{\u} - \underline{\u},\\
        %& \hat{\bm{W}}^{(\ell+1)} = \bm{I},\; \hat{\bm{v}}^{(\ell+1)} = \underline{\u}.\\
        \end{aligned}
    \end{equation*}
    All the other layers are identical to the FNN $\pi$, i.e., $\hat{\bm{W}}^{(k-1)} = {\bm{W}}^{(k-1)},\; \hat{\bm{v}}^{(k-1)} = {\bm{v}}^{(k-1)},\; \forall k=1,\dots,\ell-1$. 
    Then, all preceding results can be directly applied to this modified FNN. 
    %for a system with interval control input constraints. 
    However, in contrast to the method in \cite{everett2021reachability}, the two additional layers in our approach do not induce any conservatism in the reachability analysis.

\section{Safety Verification for Neural Feedback Systems via BRS}

In this section, the backward reachability analysis in the preceding section will be utilized for the safety verification of neural feedback systems.

Consider an initial state set $\mathcal{X}_0 \subset \mathcal{X}$ and an unsafe region $\mathcal{O}\subset \mathcal{X}$, both of which are represented as HZs. We consider the unsafe set $\mathcal{O}$ as the target set in Section \ref{sec:main} and suppose that the exact $t$-step BRS of $\mathcal{O}$ can be computed as $\mathcal{P}_t(\mathcal{O})$ by \eqref{equ:t-step}, where $t=1,\dots,T$ with $T$ an arbitrary positive integer. Clearly, if $\mathcal{X}_0$ does not intersect with any $\mathcal{P}_t$ for $t=1,\dots,T$, any state trajectory that starts from $\mathcal{X}_0$ will not enter into the unsafe region $\mathcal{O}$ within $T$ time steps, in other words, the neural feedback system \eqref{close-sys} is safe within $T$ steps. By \cite[Proposition 7]{bird2021hybrid} and Lemma \ref{lemma:empty}, checking the emptiness of the intersection of $\mathcal{X}_0$ and $\mathcal{P}_t$ is equivalent to solving an MILP. 

The safety verification of neural feedback systems via BRSs is summarized in the following proposition whose proof is omitted due to space limitations.
%Therefore, the safety verification of neural feedback systems via BRSs is equivalent to solving a MILP, as shown in the following result.
%and for $t = 1,\dots,T$ 
%Proposition \ref{prop:verify} is a straightforward application of Lemma \ref{lemma:empty} on the intersection of $\mathcal{P}_t$ and $\mathcal{X}_0$ for $t = 1,\dots,T$.

%Then, by checking the intersection between the BRSs $\{\mathcal{P}_t(\mathcal{O})\}_{t=1}^T$ and the initial set $\mathcal{X}_0$, we can certify the safety of the neural feedback system  in the sense that the state trajectories of the neural feedback system \eqref{close-sys} can avoid the unsafe region $\mathcal{O}$ in $T$ steps. This safety verification problem is equivalent to solving a MILP, as shown in the following result.
%\XX{finish. explain in what sense the system is called safe.}. 

%Let the exact $t$-step BRS computed by \eqref{equ:t-step} be $\mathcal{P}_t(\mathcal{O}) = \hz{t}$, for $t=1,\dots,T$. Denote $\mathcal{X}_0 = \hz{0}\subset \mathcal{X}$ as the initial state set. 

%The following result provides a sufficient and necessary condition formulated as a MILP to certify the safety of the neural feedback system \eqref{close-sys}.

\begin{proposition}\label{prop:verify}
Suppose that an initial state set $\mathcal{X}_0 = \hz{0}\subset \mathcal{X}$ and an unsafe set $\mathcal{O}\subset \mathcal{X}$ are both HZs, and $\mathcal{P}_t = \hz{t}$ is the exact $t$-step BRS of $\mathcal{O}$ where $t=1,\dots,T$ with $T$ an arbitrary positive integer. Then, the state trajectories of the neural feedback system \eqref{close-sys} starting from $\mathcal{X}_0$ can avoid the unsafe region $\mathcal{O}$ within $T$ steps, if and only if the following condition holds for $t = 1,\dots,T$:
\begin{equation}\label{equ:verify}
    \begin{aligned}
        \!\!\min &\left\{ \|\bm{\xi}^c\|_{\infty} \left | \mat{\A_t^c \!\!& \!\! \mathbf{0}\\ \mathbf{0} \!\!& \!\! \A_0^c \\ \G_t^c \!\!& \!\! -\G_0^c}\bm{\xi}^c  \!+\! \mat{\A_t^b \!\!& \!\! \mathbf{0}\\ \mathbf{0} \!\!& \!\! \A_0^b \\ \G_t^b \!\!& \!\! -\G_0^b}\bm{\xi}^b\right.\right.\\ 
&\left.= \! \mat{\bb_t\\ \bb_0\\ \cc_0 - \cc_t},\bm{\xi}^c\in\mathbb{R}^{n_{g,t}}, \bm{\xi}^b\in\{-1,1\}^{n_{b,t}} \right\} \!>\! 1.
    \end{aligned}
\end{equation}
\end{proposition}

%Proposition \ref{prop:verify} is a straightforward application of Lemma \ref{lemma:empty} on the intersection of $\mathcal{P}_t$ and $\mathcal{X}_0$ for $t = 1,\dots,T$. 
%The proof of this  is omitted due to space limitation.

\begin{remark}\label{rmk:relax}
    Denote the number of continuous generators, binary generators and equality constraints of the HZ $\mathcal{O}$ (resp. $\mathcal{X}_0$) as $n_{g,o}$, $n_{b,o}$ and $n_{c,o}$ (resp. $n_{g,0}$, $n_{b,0}$ and $n_{c,0}$), respectively. 
    %Safety verification of the neural feedback system via solving 
    The $T$ MILPs in \eqref{equ:verify} include $n_{g,t}$ continuous variables, $n_{b,t}$ binary variables, and $n_{c,t}$ linear constraints, where $n_{g,t} = t\cdot(n_{g,x}+6N_\pi) + n_{g,o}+n_{g,0}$, $n_{b,t} = t\cdot(n_{b,x}+N_\pi) + n_{b,o}+n_{b,0}$ and $n_{c,t} = t\cdot(n_{c,x}+5N_\pi+n) + n_{c,o}+n_{c,0}+n$.
    Commercial solvers such as Gurobi \cite{gurobi} have shown promising performance in solving MILPs. 
    %Although commercial solvers such as Gurobi \cite{gurobi} have shown promising performance in solving MILPs, they are known to be NP-hard problems and don't generalize well for a large number of binary variables,i.e., $n_{b,t}$. 
    To further reduce the computation burden, we can use Lemma 5 in \cite{zhang2022reachability} to get the tightest convex relaxation of the exact BRS $\mathcal{P}_t$ by replacing the binary generators with continuous generators. If relaxed BRSs are used in Proposition \ref{prop:verify}, \eqref{equ:verify} will degenerate into linear programs which are much easier to solve. 
\end{remark}

\section{Simulation Examples}\label{sec:sim}
In this section, 
two simulation examples will be presented to 
%a learned damped pendulum model and a double integrator model are simulated to 
demonstrate the effectiveness of the proposed method. The method proposed in this work is implemented in MATLAB R2022a and executed on a desktop with an Intel Core i9-12900k CPU and 16GB of RAM.

% The method proposed in this work is implemented in MATLAB R2022a on a desktop with an Intel Core i9-12900k CPU and 16GB of RAM.

%\subsection{Damped Pendulum Model}
\begin{example}[Damped Pendulum Model]\label{example1}
Consider the damped pendulum model given in \cite{vincent2021reachable}. A fully-connected FNN with ReLU activation functions and one hidden layer of 12 neurons was trained to approximate the discrete-time dynamics of the pendulum. The learned dynamics is
$
    \mt{x}(t+1) = \f_{NN}(\mt{x}(t))
$ 
where $\mt{x} = [\mt{\theta}, \dot{\mt{\theta}}]^\top$ and $\f_{NN}(\cdot)$ is the trained FNN. We chose the target set as $\mathcal{T} = [-10,10]\times[-30,30]$, the state set as $\mathcal{X} = [-90,90]\times[-90,90]$, and $\alpha = \beta = 100$.

Using Theorem \ref{thm:brs} and equations \eqref{equ:t-step}, 50 exact BRSs  $\mathcal{P}_{1}(\mathcal{T})$, $\mathcal{P}_{2}(\mathcal{T})$, $\dots,\mathcal{P}_{50}(\mathcal{T})$ were computed within 0.891 seconds. Figure \ref{fig:pendulum} illustrates the set $\mathcal{P}_{5}(\mathcal{T})$, which is the union of 31 polytopes, in the $\theta$-$\dot\theta$ plane. Although $\mathcal{P}_{5}(\mathcal{T})$ computed using \eqref{equ:t-step} has 60 binary variables, only 31 binary value combinations satisfy the linear equality constraints of the HZ, which correspond to the resulting 31 polytopes. 
%To verify the exactness of our computed BRSs, 
We also ran the RPM Algorithm developed in \cite{vincent2021reachable} which produced the same exact BRSs as our method. The computation time of the RPM method in Julia is 69.844 seconds for $T = 50$.

\begin{figure}[t]
    \centering
    \includegraphics[width=\linewidth]{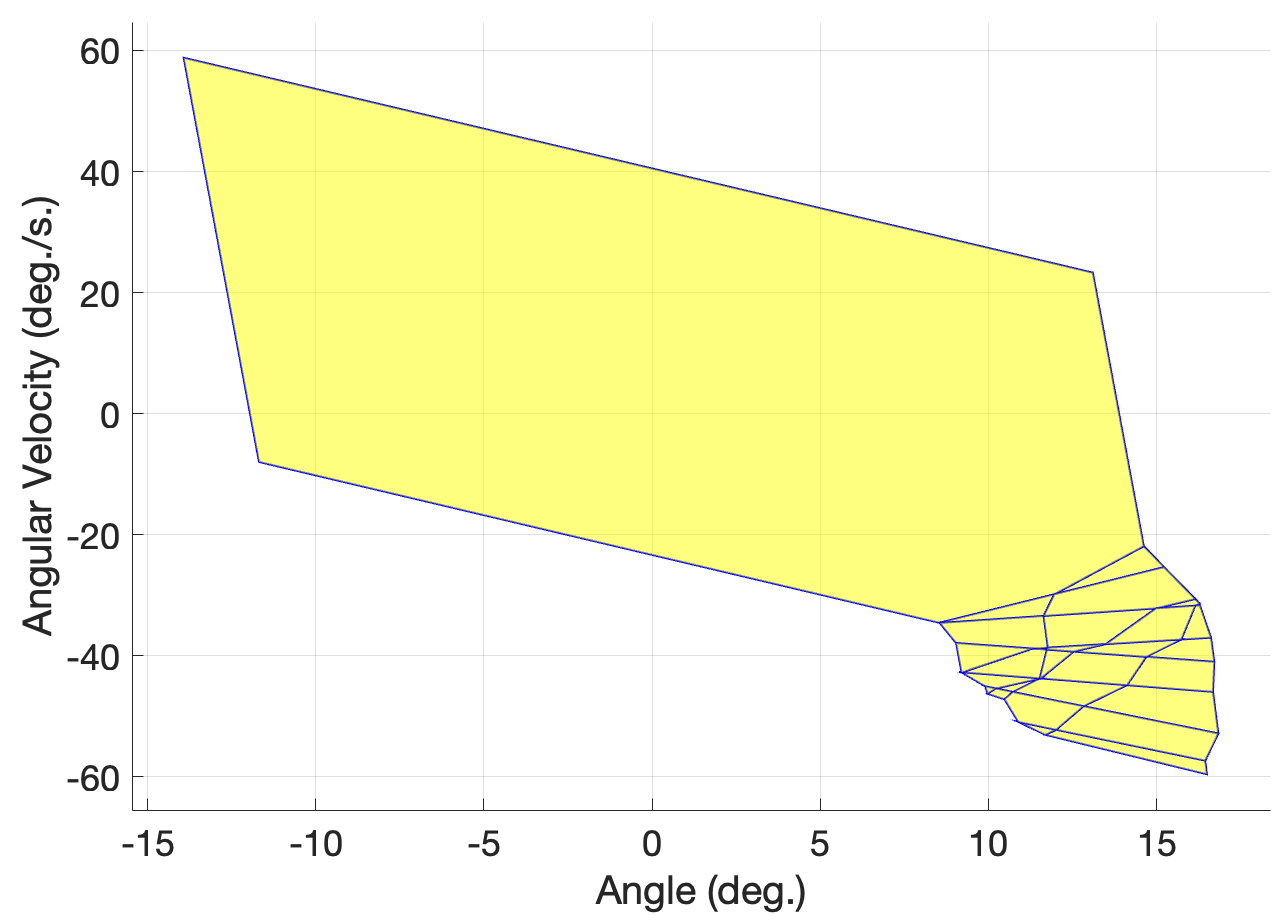}
    % \caption{Set of states in the 5-step BRS for damped pendulum model, which will be mapped to the target set $\mathcal{T} = [-10,10]\times[-30,30]$ by the pendulum dynamics after 5 time steps.}
    \caption{The exact 5-step BRS $\mathcal{P}_{5}(\mathcal{T})$ in Example \ref{example1} by using our proposed method. The target set $\mathcal{T}$ is chosen as $\mathcal{T} = [-10,10]\times[-30,30]$.}
    \label{fig:pendulum}
    \vspace{-0.35cm}
\end{figure}

\end{example}

%\subsection{Double Integrator Model}
\begin{example}[Double Integrator Model]\label{example2}
Consider the discrete-time double integrator model given in \cite{Fazlyab2022Safety}:
$$
    \mt{x}(t+1) = \mat{1 & 1 \\ 0 & 1}\mt{x}(t) + \mat{0.5 \\ 1}\mt{u}(t).
$$ 
% \vspace{-3mm}
The NN controller $\u(t) = \pi(\x(t))$  has two hidden layers with ReLU activation functions and $[10, 5]$ neurons. Similar to \cite{rober2022backward}, this NN controller was trained using the dataset generated by an MPC controller. In addition, we imposed the saturation bounds $\underline{\u}=-1,\overline{\u}=1$ on the controller, i.e., $\u(t) \in\mathcal{U} = [-1,1]$. We chose the initial set as $\mathcal{X}_0 = [-1.25,0.25]\times[0.4,0.6]$, the unsafe region as  $\mathcal{O} = [4.5,5.0]\times[-0.25,0.25]$, the state region as $\mm{X} = [-40, 40]\times[-40, 40]$, and $\alpha=\beta=400$. 

We implemented Theorem \ref{thm:brs} and equations \eqref{equ:t-step} to compute 5 exact BRSs  $\mathcal{P}_{1}(\mathcal{T})$, $\mathcal{P}_{2}(\mathcal{T})$, $\dots,\mathcal{P}_{5}(\mathcal{T})$ which are shown by the sets in cyan in Figure \ref{fig:di_10_5}. We also verified that condition \eqref{equ:verify} in Proposition \ref{prop:verify} holds true, which implies the safety of the neural feedback system.  The time for computing the BRSs is 0.007 seconds, and the time for solving the MILPs given in \eqref{equ:verify} via the commercial solver Gurobi is 0.560 seconds \cite{gurobi}.

For comparison, we also ran the BReach-LP algorithm and the ReBReach-LP algorithm proposed in \cite{rober2022backward}, which were implemented in Python with default parameters provided by the authors of \cite{rober2022backward}. The computed BRSs are shown by the rectangles with orange and magenta lines in Figure \ref{fig:di_10_5}.  
%Figure \ref{fig:di_10_5} illustrates the computed BRSs of the double integrator system using different methods and uniformly generated sample trajectories. 
It can be observed that our method provides more accurate BRSs for all the time steps compared with the BReach-LP and the ReBReach-LP algorithms. In addition, the exact BRSs computed by our method certify safety in this scenario, while the over-approximated BRSs computed by the two algorithms given in \cite{rober2022backward} lead to false unsafe detection.

To verify the exactness of the BRSs by our method, we performed numerical simulations on trajectories generated from uniformly sampled initial conditions and selected the samples based on the criterion that the resulting trajectories would enter the target set within 5 steps. 
%we did numerical simulations of the trajectories starting from uniformly generated samples and selected the samples with the criterion that the corresponding trajectories will enter the target set within 5 steps. 
The selected samples are depicted by blue dots in Figure \ref{fig:di_10_5}. It can be observed that the sampled points are contained in our BRSs as expected. 

% To verify the exactness of the BRSs by our method, we compute BRSs by using uniformly generated sample trajectories where the sample points are depicted by blue dots. It can be observed that the sampled points are contained in our BRSs as expected. 

%To verify the exactness of the proposed method, The grids of the state space were generated and numerical simulations for all the grids were implemented to select the grids such that the trajectories starting from those grids will enter the target set. In addition, MILPs were solved to verify whether the initial set has intersections with BRSs and whether the selected grids are enclosed by the BRSs obtained by Algorithm \ref{alg:1} and Theorem \ref{thm:brs}. 

%Figure \ref{fig:di_10_5} shows the target set (red), the BRSs (cyan) computed by choosing $\alpha=\beta=400$ and $\mm{X} = [-40, 40]\times[-40, 40]$, the results from \cite{rober2022backward} (magenta and orange), and the selected grids lying in the true BRSs (blue scatters). All the selected grids are enclosed by the computed BRSs by solving MILPs with the commercial solver Gurobi \cite{gurobi}, which verifies the exactness of the proposed method. Moreover, there are no intersections between the initial set and all the BRSs by solving \eqref{equ:verify}, which implies that the safety of the initial set under the current NN control policy is verified. However, both BReach and ReBReach algorithms proposed in \cite{rober2022backward} cannot lead to a correct verification result due to the over-approximation.
\begin{figure}[t]
    \centering
    \includegraphics[width=\linewidth]{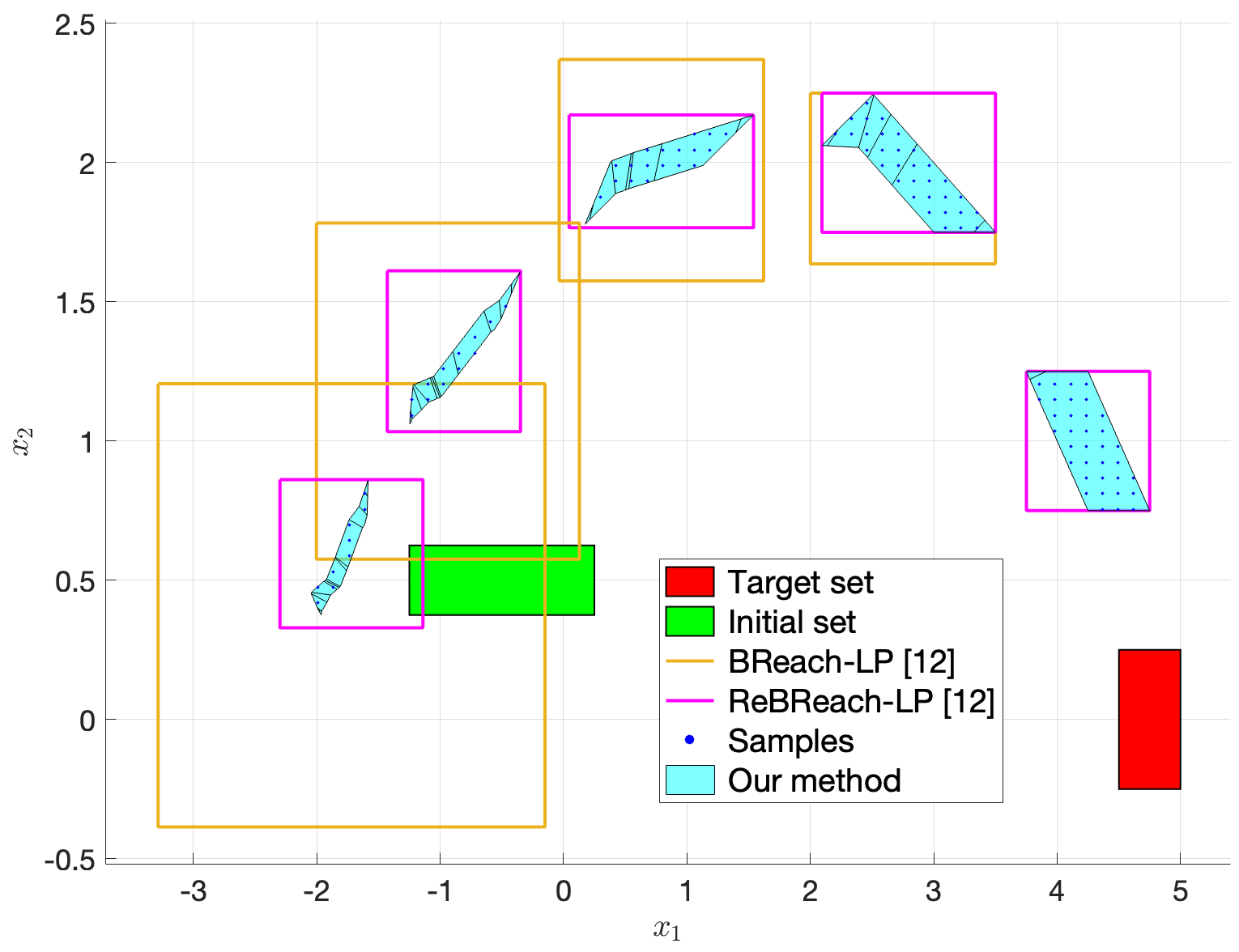}
    \caption{Simulation results in Example \ref{example2}. The exact BRSs computed by our HZ-based approach are shown in cyan. Over-approximated BRSs computed by BReach-LP and ReBReach-LP algorithms in \cite{rober2022backward} are bounded by orange and magenta lines, respectively. The target set as the unsafe region is in red and the initial set is in green. Sampled states are plotted as blue dots which are bounded by all the exact BRSs.}
    \label{fig:di_10_5}
    \vspace{-0.35cm}
\end{figure}

\end{example}

%Moreover, Table \ref{tab:1} shows the BRS computation time and the verification time of the experiment. The HZs are computed very fast since there are only matrix operations in both Algorithm \ref{alg:1} and Theorem \ref{thm:brs}. But the verification time is much slower since it is required to solve MILPs, which are NP-hard.

\section{Conclusion}\label{sec:concl}
We proposed a novel HZ-based approach to compute the exact BRSs of neural feedback systems. We showed that the input-output relationship of a ReLU-activated FNN can be exactly described by its graph set represented by an HZ. We provided an exact HZ formulation for the BRSs of neural feedback systems and extended the result to the saturated input case. We also proposed a sufficient and necessary condition in the form of MILPs for the safety verification of neural feedback systems via BRSs. The performance of the proposed approach was compared with state-of-the-art using two numerical examples.
%Future work includes extending the method to neural feedback systems with nonlinear dynamics and other types of activation functions. Additionally, order reduction techniques for hybrid zonotopes can be utilized to provide a better balance of computation complexity and efficiency.

\bibliographystyle{IEEEtran}
% \printbibliography
\bibliography{ref}

% Generated by IEEEtran.bst, version: 1.14 (2015/08/26)
\begin{thebibliography}{10}
\providecommand{\url}[1]{#1}
\csname url@samestyle\endcsname
\providecommand{\newblock}{\relax}
\providecommand{\bibinfo}[2]{#2}
\providecommand{\BIBentrySTDinterwordspacing}{\spaceskip=0pt\relax}
\providecommand{\BIBentryALTinterwordstretchfactor}{4}
\providecommand{\BIBentryALTinterwordspacing}{\spaceskip=\fontdimen2\font plus
\BIBentryALTinterwordstretchfactor\fontdimen3\font minus
  \fontdimen4\font\relax}
\providecommand{\BIBforeignlanguage}[2]{{%
\expandafter\ifx\csname l@#1\endcsname\relax
\typeout{** WARNING: IEEEtran.bst: No hyphenation pattern has been}%
\typeout{** loaded for the language `#1'. Using the pattern for}%
\typeout{** the default language instead.}%
\else
\language=\csname l@#1\endcsname
\fi
#2}}
\providecommand{\BIBdecl}{\relax}
\BIBdecl

\bibitem{yuan2019adversarial}
X.~Yuan, P.~He, Q.~Zhu, and X.~Li, ``Adversarial examples: Attacks and defenses
  for deep learning,'' \emph{IEEE Transactions on Neural Networks and Learning
  Systems}, vol.~30, no.~9, pp. 2805--2824, 2019.

\bibitem{goodfellow2014explaining}
I.~J. Goodfellow, J.~Shlens, and C.~Szegedy, ``Explaining and harnessing
  adversarial examples,'' in \emph{Proceedings of the 3rd International
  Conference on Learning Representations}, 2015.

\bibitem{bekey2012neural}
G.~A. Bekey and K.~Y. Goldberg, \emph{Neural Networks in Robotics}.\hskip 1em
  plus 0.5em minus 0.4em\relax Springer Science \& Business Media, 2012, vol.
  202.

\bibitem{rao2018deep}
Q.~Rao and J.~Frtunikj, ``Deep learning for self-driving cars: Chances and
  challenges,'' in \emph{Proceedings of the 1st International Workshop on
  Software Engineering for AI in Autonomous Systems}, 2018, pp. 35--38.

\bibitem{katz2017reluplex}
G.~Katz, C.~Barrett, D.~L. Dill, K.~Julian, and M.~J. Kochenderfer, ``Reluplex:
  An efficient {SMT} solver for verifying deep neural networks,'' in \emph{29th
  International Conference on Computer Aided Verification}.\hskip 1em plus
  0.5em minus 0.4em\relax Springer, 2017, pp. 97--117.

\bibitem{dutta2019reachability}
S.~Dutta, X.~Chen, and S.~Sankaranarayanan, ``Reachability analysis for neural
  feedback systems using regressive polynomial rule inference,'' in
  \emph{Proceedings of the 22nd ACM International Conference on Hybrid Systems:
  Computation and Control}, 2019, pp. 157--168.

\bibitem{huang2019reachnn}
C.~Huang, J.~Fan, W.~Li, X.~Chen, and Q.~Zhu, ``Reach{NN}: Reachability
  analysis of neural-network controlled systems,'' \emph{ACM Transactions on
  Embedded Computing Systems}, vol.~18, no.~5s, pp. 1--22, 2019.

\bibitem{everett2021reachability}
M.~Everett, G.~Habibi, C.~Sun, and J.~P. How, ``Reachability analysis of neural
  feedback loops,'' \emph{IEEE Access}, vol.~9, pp. 163\,938--163\,953, 2021.

\bibitem{tran2020nnv}
H.-D. Tran, X.~Yang, D.~Manzanas~Lopez, P.~Musau, L.~V. Nguyen, W.~Xiang,
  S.~Bak, and T.~T. Johnson, ``{NNV}: The neural network verification tool for
  deep neural networks and learning-enabled cyber-physical systems,'' in
  \emph{32nd International Conference on Computer-Aided Verification}.\hskip
  1em plus 0.5em minus 0.4em\relax Springer, 2020, pp. 3--17.

\bibitem{Fazlyab2022Safety}
M.~Fazlyab, M.~Morari, and G.~J. Pappas, ``Safety verification and robustness
  analysis of neural networks via quadratic constraints and semidefinite
  programming,'' \emph{IEEE Transactions on Automatic Control}, vol.~67, no.~1,
  pp. 1--15, 2022.

\bibitem{zhang22constrainedzonotope}
Y.~Zhang and X.~Xu, ``Safety verification of neural feedback systems based on
  constrained zonotopes,'' in \emph{IEEE 61st Conference on Decision and
  Control}, 2022, pp. 2737--2744.

\bibitem{rober2022backward}
N.~Rober, M.~Everett, and J.~P. How, ``Backward reachability analysis for
  neural feedback loops,'' in \emph{IEEE 61st Conference on Decision and
  Control}, 2022, pp. 2897--2904.

\bibitem{rober2022hybrid}
N.~Rober, M.~Everett, S.~Zhang, and J.~P. How, ``A hybrid partitioning strategy
  for backward reachability of neural feedback loops,'' \emph{arXiv preprint
  arXiv:2210.07918}, 2022.

\bibitem{vincent2021reachable}
J.~A. Vincent and M.~Schwager, ``Reachable polyhedral marching ({RPM}): A
  safety verification algorithm for robotic systems with deep neural network
  components,'' in \emph{IEEE International Conference on Robotics and
  Automation}, 2021, pp. 9029--9035.

\bibitem{mitchell2005time}
I.~Mitchell, A.~Bayen, and C.~Tomlin, ``A time-dependent {H}amilton-{J}acobi
  formulation of reachable sets for continuous dynamic games,'' \emph{IEEE
  Transactions on Automatic Control}, vol.~50, no.~7, pp. 947--957, 2005.

\bibitem{althoff2021set}
M.~Althoff, G.~Frehse, and A.~Girard, ``Set propagation techniques for
  reachability analysis,'' \emph{Annual Review of Control, Robotics, and
  Autonomous Systems}, vol.~4, pp. 369--395, 2021.

\bibitem{yang2022efficient}
L.~Yang, H.~Zhang, J.-B. Jeannin, and N.~Ozay, ``Efficient backward
  reachability using the {M}inkowski difference of constrained zonotopes,''
  \emph{IEEE Transactions on Computer-Aided Design of Integrated Circuits and
  Systems}, vol.~41, no.~11, pp. 3969--3980, 2022.

\bibitem{bird2021hybrid}
T.~J. Bird, H.~C. Pangborn, N.~Jain, and J.~P. Koeln, ``Hybrid zonotopes: A new
  set representation for reachability analysis of mixed logical dynamical
  systems,'' \emph{arXiv preprint arXiv:2106.14831}, 2021.

\bibitem{siefert2022robust}
J.~A. Siefert, T.~J. Bird, J.~P. Koeln, N.~Jain, and H.~C. Pangborn, ``Robust
  successor and precursor sets of hybrid systems using hybrid zonotopes,''
  \emph{IEEE Control Systems Letters}, vol.~7, pp. 355--360, 2022.

\bibitem{bird2022set}
T.~J. Bird, N.~Jain, H.~C. Pangborn, and J.~P. Koeln, ``Set-based reachability
  and the explicit solution of linear {MPC} using hybrid zonotopes,'' in
  \emph{American Control Conference}.\hskip 1em plus 0.5em minus 0.4em\relax
  IEEE, 2022, pp. 158--165.

\bibitem{zhang2022reachability}
\BIBentryALTinterwordspacing
Y.~Zhang and X.~Xu, ``Reachability analysis and safety verification of neural
  feedback systems via hybrid zonotopes,'' in \emph{American Control
  Conference}.\hskip 1em plus 0.5em minus 0.4em\relax IEEE, 2023 (to appear).
  [Online]. Available: \url{https://arxiv.org/abs/2210.03244}
\BIBentrySTDinterwordspacing

\bibitem{bird2021unions}
T.~J. Bird and N.~Jain, ``Unions and complements of hybrid zonotopes,''
  \emph{IEEE Control Systems Letters}, vol.~6, pp. 1778--1783, 2021.

\bibitem{bird2022hybrid}
T.~J. Bird, ``Hybrid zonotopes: A mixed-integer set representation for the
  analysis of hybrid systems,'' \emph{Purdue University Graduate School}, 2022.

\bibitem{keerthi1987computation}
S.~Keerthi and E.~Gilbert, ``Computation of minimum-time feedback control laws
  for discrete-time systems with state-control constraints,'' \emph{IEEE
  Transactions on Automatic Control}, vol.~32, no.~5, pp. 432--435, 1987.

\bibitem{blanchini2008set}
F.~Blanchini and S.~Miani, \emph{Set-Theoretic Methods in Control}.\hskip 1em
  plus 0.5em minus 0.4em\relax Springer, 2008, vol.~78.

\bibitem{anevlavis2019computing}
T.~Anevlavis and P.~Tabuada, ``Computing controlled invariant sets in two
  moves,'' in \emph{IEEE 58th Conference on Decision and Control}, 2019, pp.
  6248--6254.

\bibitem{yin2021stability}
H.~Yin, P.~Seiler, and M.~Arcak, ``Stability analysis using quadratic
  constraints for systems with neural network controllers,'' \emph{IEEE
  Transactions on Automatic Control}, vol.~67, no.~4, pp. 1980--1987, 2021.

\bibitem{rossig2021advances}
A.~R{\"o}ssig and M.~Petkovic, ``Advances in verification of {R}e{LU} neural
  networks,'' \emph{Journal of Global Optimization}, vol.~81, pp. 109--152,
  2021.

\bibitem{gurobi}
\BIBentryALTinterwordspacing
{Gurobi Optimization, LLC}, ``Gurobi optimizer reference manual,'' 2022.
  [Online]. Available: \url{https://www.gurobi.com}
\BIBentrySTDinterwordspacing

\end{thebibliography}

\end{document}